\documentclass[11 pt, a4paper, oneside]{amsart}

\usepackage[utf8]{inputenc}
\usepackage{cmap}
\usepackage[T1]{fontenc}
\usepackage[top=3cm,bottom=2cm,right=2cm,left=3cm]{geometry}
\usepackage{color}
\usepackage{indentfirst}
\usepackage{lmodern} \normalfont
\DeclareFontShape{T1}{lmr}{bx}{sc} { <-> ssub * cmr/bx/sc }{}
\usepackage{amsmath}
\usepackage{amsfonts}
\usepackage{amssymb}
\usepackage{amsrefs}

\usepackage{amsmath}
\usepackage{amsthm,amsfonts,mathrsfs,amsfonts}
\usepackage{slashed}
\usepackage{breqn}
\usepackage{pb-diagram}
\usepackage{bbm}
\usepackage[all]{xy}
\usepackage{times}
\usepackage{amsaddr}
\usepackage{microtype}
\usepackage{enumerate}
\usepackage[super]{nth}
\usepackage{multicol}
\usepackage{tikz-cd}

\usepackage{url}

\newtheorem{theorem}{Theorem}[section]
\newtheorem{corollary}[theorem]{Corollary}

\newtheorem{lemma}[theorem]{Lemma}
\newtheorem{remark}[theorem]{Remark}
\newtheorem*{remark*}{Remark}

\newcommand{\gt}{\mathrm{G}_2}
\DeclareMathOperator{\vol}{vol}

\DeclareMathOperator{\Gl}{Gl}

\DeclareMathOperator{\SU}{SU}
\DeclareMathOperator{\End}{End}

\DeclareMathOperator{\Ad}{Ad}

\DeclareMathOperator{\Aut}{Aut}
\DeclareMathOperator{\Der}{Der}

\DeclareMathOperator{\diag}{diag}
\DeclareMathOperator{\sym}{sym}

\DeclareMathOperator{\ricci}{Ric}
\DeclareMathOperator{\expo}{exp}
\newcommand{\symtensor}{\mathrm{S}^2(T^\ast M)}
\DeclareMathOperator{\traceless}{S^2_0(T^\ast M)}

\newcommand{\R}{\mathbb{R}}

\newcommand{\qforq}{\quad \text{for} \quad}
\newcommand{\qandq}{\quad \text{and} \quad}

\makeatletter
\def\@tocline#1#2#3#4#5#6#7{\relax
  \ifnum #1>\c@tocdepth 
  \else
    \par \addpenalty\@secpenalty\addvspace{#2}%
    \begingroup \hyphenpenalty\@M
    \@ifempty{#4}{%
      \@tempdima\csname r@tocindent\number#1\endcsname\relax
    }{%
      \@tempdima#4\relax
    }%
    \parindent\z@ \leftskip#3\relax \advance\leftskip\@tempdima\relax
    \rightskip\@pnumwidth plus4em \parfillskip-\@pnumwidth
    #5\leavevmode\hskip-\@tempdima
      \ifcase #1
       \or\or \hskip 2em \or \hskip 3em \else \hskip 4em \fi%
      #6\nobreak\relax
    \dotfill\hbox to\@pnumwidth{\@tocpagenum{#7}}\par
    \nobreak
    \endgroup
  \fi}
\makeatother

\begin{document}

\title{Explicit soliton for the Laplacian co-flow on a solvmanifold}
\author{Andrés J. Moreno \quad \& \quad Henrique N. S\'a Earp}
\address{University of Campinas (Unicamp)}
\date{\today}

\begin{abstract}
We apply the general Ansatz proposed by Lauret  \cite{Lauret} for the Laplacian co-flow of invariant $\gt$-structures on a Lie group, finding  an explicit soliton  on a particular almost Abelian $7$--manifold. Our methods and the example itself are different from those presented by Bagaglini and Fino \cite{Fino}.
\end{abstract}

\maketitle

%

\tableofcontents

\section*{Introduction}

Geometric flows in $\gt$-geometry were first outlined by the seminal works of Bryant \cite{Bryan2} and Hitchin \cite{Hitchin2001}, and have since been studied by several authors, e.g. \cite{Bryan,Fino,Grigorian,Karigianis,Lauret}. These so-called \emph{$\gt$-flows} arise as a tool in the search for ultimately torsion-free $\gt$-structures, by varying  a non-degenerate $3$-form on an oriented and spin    $7$--manifold $M$ towards some   $\varphi\in \Omega^3:=\Omega^3(M)$ such that the \emph{torsion } $\nabla^{g_\varphi}\varphi$ vanishes, where $g_\varphi$ is the natural Riemannian metric defined from $\varphi$ by
$$
g_\varphi(X,Y) \cdot d \mathrm{Vol} 
:=(X\lrcorner \varphi)\wedge (Y\lrcorner\varphi)\wedge\varphi.
$$
Such pairs $(M^7,\varphi)$ solving the nonlinear PDE problem  $\nabla^{g_\varphi}\varphi\equiv0$ are called \emph{$\gt$--manifolds}  and  are very difficult to construct, especially when  $M$ is required to be compact. To this date, all known solutions stem from elaborate constructions in geometric analysis \cite{Joyce,corti2015,Eguchi-Hanson}. 

Some weaker formulations of that problem can be obtained from the  classical fact, first established by Fernández and Gray \cite{Fernandez},  that the torsion-free condition is equivalent to $\varphi$ being both \emph{closed} and \emph{coclosed}, in the sense that $d\varphi=0$ and $d\ast_\varphi\varphi=0$, respectively, and thus one may study each of these conditions separately.
For instance, Grigorian \cite{Grigorian} and Karigiannis and Tsui  \cite{Karigianis} considered the \emph{Laplacian co-flow} of $\gt$-structures $\{\varphi_t\}$ defined by 
\begin{equation}
\label{Laplacia_coflow}
        \frac{\partial\psi_t}{\partial t}
        =-\Delta_{\psi_t}\psi_t,
\end{equation}
where $\psi_t:=\ast_t\varphi_t$ is the Hodge dual and $\Delta_{\psi_t}\psi_t:=(dd^{\ast_t}+d^{\ast_t}d)\psi_t$  is the Hodge Laplacian of the metric $g_{\varphi_t}$ on $4$--forms. It is a natural process to consider among  coclosed $\gt$--structures, as it manifestly preserves that property, i.e.,  it flows $\psi_t$ in its de Rham cohomology class. Moreover, it is the gradient flow of Hitchin's volume functional \cite{Hitchin2001}.

When $M^7=G$ is a Lie group, we propose to study this flow from the perspective introduced by Lauret \cite{Lauret} in the general context of geometric flows on homogeneous spaces. As a proof of principle, we apply a natural Ansatz to construct an example of invariant self-similar solution, or \emph{soliton}, of the Laplacian co-flow.
Solitons are $\gt$--structures which, under the flow, simply scale monotonically and move by diffeomorphisms.  In particular, they provide potential models for singularities of the flow, as well as means for desingularising certain singular $\gt$--structures, both of which are key aspects of any geometric flow.
We follow in spirit the approach of 
Karigiannis et al.~\cite{Karigianis} to obtain  solitons to the Laplacian coflow  from a general Ansatz for a coclosed  cohomogeneity one $\gt$--structure on manifolds of the form $M^ 7 = N^ 6\times L^ 1$, where $L^1=\R$ or  $S^1$ and  $N^6$ is compact and either nearly K\"ahler or a Calabi-Yau $3$-fold.
In that case, as in ours, the symmetries of the  space are  exploited to reduce the soliton condition to a manageable ODE. 

\section{Torsion forms of a $\gt$-structure}
Let us briefly review some elementary representation theory underlying  $\gt$-geometry, following the setup from \cite{Bryan2,Karigianis2}. The natural action of $G_2\subset SO(7)$ decomposes $\Omega^\bullet(M)$ into $\gt$-invariant irreducible subbundles: 
\begin{equation}\label{forms_split}
\begin{aligned}
\Omega^1 &=\Omega^1_7, &        \Omega^2 &= \Omega^2_7\oplus \Omega^2_{14}, & \Omega^3 &=\Omega^3_1\oplus\Omega^3_7\oplus\Omega^3_{27},\\
\Omega^6 &=\Omega^6_7, &        \Omega^5 &= \Omega^5_7\oplus \Omega^5_{14}, & \Omega^4 &=\Omega^4_1\oplus\Omega^4_7\oplus\Omega^4_{27}, 
\end{aligned}
\end{equation}
where each $\Omega^k_l$ has rank $l$. Studying the symmetries of torsion one finds that $\nabla\varphi\in \Omega^1\otimes\Omega^3_7$, so that tensor lies in a bundle of rank  $49$ \cite[Lemma 2.24]{Karigianis2}. Notice also that $\Omega^3_7\cong\Omega^1$, so, contracting the dual $4$-form $\psi=\ast_\varphi\varphi$ by a frame of $TM$, then using the Riemannian metric, one has
\begin{equation*}
        \Omega^2\oplus \symtensor=\Omega^1\otimes\Omega^3_7\cong \End(TM)=\mathfrak{so}(TM)\oplus\sym(TM). 
\end{equation*}
Here $\symtensor$ denotes the symmetric bilinear forms and $\sym(TM)$ the symmetric endomorphisms of $TM$. Both of the  above splittings are $\gt$-invariant, so, comparing the $\gt$-irreducible decomposition $\mathfrak{so}(7)=\mathfrak{g}_2\oplus [\mathbb{R}^7]$ and \eqref{forms_split}, we get the following identification between $\gt$-irreducible summands 
\begin{equation}
[\mathbb{R}^7]\cong \Omega^2_7 \qandq \mathfrak{g}_2\cong \Omega^2_{14}.
\end{equation}
For  $\symtensor\cong \sym(TM)$, Bryant defines  maps $i: \symtensor \rightarrow \Omega^3$ and $j:\Omega^3\rightarrow \symtensor$ by 
\begin{equation}\label{Bryan_map}
        i(h)=\frac{1}{2}h_{il}g^{lm}\varphi_{mjk}dx^{ijk} \qandq j(\eta)(u,v)=\ast((u\lrcorner\varphi)\wedge(v\lrcorner\varphi)\wedge\eta),
\end{equation}
where we adopt the familiar  implicit summation convention for repeated indices and the inverse of the metric. The map $i$ is injective \cite[Corollary 2.16]{Karigianis2} and, by the $\gt$-decomposition $\symtensor=\mathbb{R}g_\varphi\oplus \traceless$, it identifies
$$
\mathbb{R}g_\varphi\cong\Omega^3_1 \qandq \traceless\cong\Omega^3_{27}.
$$ 
Accordingly, we have a  decomposition for the torsion components $d\varphi\in\Omega^4$ and $d\psi\in\Omega^5$ given by
$$
  d\varphi=\tau_0\psi+3\tau_1\wedge\varphi+\ast\tau_3 \qandq d\psi=4\tau_1\wedge\psi+\ast\tau_2,
$$
where $\tau_0\in\Omega^0$, $\tau_1\in\Omega^1$, $\tau_2\in\Omega^2_{14}$ and $\tau_3\in\Omega^3_{27}$ are called the \emph{torsion forms}. Indeed, the torsion is completely encoded in the \emph{full torsion tensor} $T$, defined in coordinates by 
$$
\nabla_l\varphi_{abc}=:T_{lm}g^{mn}\psi_{nabc},
$$
which is expressed in terms of the irreducible $\gt$-decomposition of $\End(TM)$ by \cite[Theorem 2.27]{Karigianis2}  
$$
  T=\frac{\tau_0}{4}g_\varphi-\tau_{27}+(\tau_1)^\sharp\lrcorner\varphi-\frac{1}{2}\tau_2,
$$ 
where $\tau_3:=i(\tau_{27})$ and $ ^\sharp: \Omega^1 \rightarrow \mathcal{X}(M)$ the musical isomorphism induced by the $\gt$-metric. 
If moreover the $\gt$-structure is co-closed, the torsion tensor  $T=\frac{\tau_0}{4}g_\varphi-\tau_{27}$ is totally symmetric, and the Hodge Laplacian of $\psi$ is given by \cite{Heisenberg}
$$
  \Delta_\psi\psi=dd^{\ast}\psi=d\tau_0\wedge\varphi+\tau_0^2\psi+\tau_0\ast\tau_3+d\tau_3.
$$
If moreover $\tau_3$ vanishes, then  $\psi$ is a Laplacian eigenform and the $\gt$-structure is called \emph{nearly parallel}.  

\section{Invariant $\gt$-structures on Lie groups}

Let us briefly survey Lauret's approach  to geometric flows on homogeneous spaces  \cite{Lauret}. 
for $X_1,...,X_r\in\Gamma(TM)$ and $\alpha_1,\dots,\alpha_s\in\Gamma(T^\ast M)$. In particular, when $M=G/H$ is a reductive homogeneous space,  i.e.
$$
 \mathfrak{g}=\mathfrak{h}\oplus\mathfrak{m} \quad \text{such that} \quad \Ad(h)\mathfrak{m}\subset \mathfrak{m}, \ \forall h\in H,
$$
any $G$-invariant tensor $\gamma$ is completely determined by its value $\gamma_{x_0}$ at the point $x_0=[1_G]\in G/H$, where $\gamma_{x_0}$ is an $\Ad(H)$-invariant tensor at $\mathfrak{m}\cong T_{x_0}M$, i.e. $(\Ad(h))^\ast\gamma_{x_0}=\gamma_{x_0}$ for each $h\in H$.
Given $x=[gx_0]\in G/H$, clearly  $\gamma_x=(g^{-1})^\ast\gamma_{x_0}$.
Consider now a  geometric flow on $M$ of the  general form
\begin{equation}\label{Lauret_flow}
        \frac{\partial}{\partial t}\gamma_t=q(\gamma_t).
\end{equation}
Then, if $M=G/H$, requiring $G$-invariance of $\gamma_t$, for all $t$, reduces the flow to an ODE for a one-parameter family $\gamma_t$ of $\Ad(H)$-invariant tensors on the vector space $\mathfrak{m}$:
\begin{equation*}
        \frac{d}{dt}\gamma_t=q(\gamma_t).
\end{equation*}

Now, we fix $\dim G=7$ and $H=\{1\}$ the trivial subgroup. For any $\gt$-structure   $\varphi_0\in \Lambda^3(\mathfrak{g})^\ast$     on $\mathfrak{g}=\mathrm{Lie}(G)$, non-degeneracy means that, for each non-zero vector $v\in \mathfrak{g}$, the $2$-form $\iota(v)\varphi_0$ is symplectic  on the vector space $\mathfrak{g}/\left\langle v \right\rangle$. We also know that the $\Gl(\mathfrak{g})$-orbit of the dual $4$-form $\psi_0=\ast_0\varphi_0$ is open in $\Lambda^4(\mathfrak{g})^\ast$ under the natural action 
$$
h\cdot\psi_0:=(h^{-1})^\ast\psi_0=\psi_0(h^{-1}\cdot,h^{-1}\cdot,h^{-1}\cdot,h^{-1}\cdot), \quad h\in \Gl(\mathfrak{g}). 
$$
Denoting by  $\theta: \mathfrak{gl}(\mathfrak{g})\rightarrow \End(\Lambda^4(\mathfrak{g})^\ast)$ the infinitesimal representation  $\theta(A)\psi_0:=\frac{d}{dt}(e^{tA}\cdot\psi_0)|_{t=0}$, we have 
\begin{equation}\label{local_model}
\theta(\mathfrak{gl}(\mathfrak{g}))\psi_0=\Lambda^4(\mathfrak{g})^\ast,
\end{equation}
 and the Lie algebra of the stabilizer 
$$
G_{\psi_0}:=\{h\in\Gl(\mathfrak{g})  \ ; \ h\cdot\psi_0=\psi_0  \}
$$
is characterised by  
$$
 \mathfrak{g}_{\psi_0}:=\mathrm{Lie}(G_{\psi_0})=\{A\in\mathfrak{gl}(\mathfrak{g}) \ ; \ \theta(A)\psi_0=0 \}.
$$
Indeed, the orbit $\Gl(\mathfrak{g})\cdot\psi_0$ is parametrised by the homogeneous space $\Gl(\mathfrak{g})/G_{\psi_0}$. Using the reductive decomposition $\mathfrak{gl}(\mathfrak{g})=\mathfrak{g}_{\psi_0}\oplus\mathfrak{q}_{\psi_0}$ from equation \eqref{local_model}, we have
\begin{equation}
        \theta(\mathfrak{q}_{\psi_0})\psi_0=\Lambda^4(\mathfrak{g})^\ast.
\end{equation}
In particular, for the Laplacian $\Delta_0\psi_0$, there exists a unique $Q_0\in \mathfrak{q}_{\psi_0}$ such that $\theta(Q_0)\psi_0=\Delta_0\psi_0$. Now, for any other $\psi=h\cdot\psi_0\in \Gl(\mathfrak{h})\cdot\psi_0$, 
$$
G_\psi=G_{h\cdot\psi_0}=h^{-1}G_{\psi_0}h \qandq \mathfrak{g}_{\psi}=\mathfrak{g}_{h\cdot\psi_0}=\Ad(h^{-1})\mathfrak{g}_{\psi_0},
$$
where $\Ad: \Gl(\mathfrak{g})\rightarrow \Gl(\mathfrak{gl}(\mathfrak{g}))$. Moreover, we have the following relations.
\begin{lemma}\label{diffeomorphism_invariance}
        Let $\psi= h\cdot \psi_0$ for $h\in \Gl(\mathfrak{g})$, denote $\ast$ the Hodge star and $\Delta$ the Laplacian operator of $\psi$, then
        $$
          \ast=(h^{-1})^\ast\ast_0h^\ast \qandq h^\ast\circ\Delta=\Delta_0\circ h^\ast,
        $$
        where $\ast_0$ and $\Delta_0$ are the Hodge star and the Laplacian operator of $\psi_0$, respectively.
\end{lemma}

\begin{proof}
        The inner products on $\mathfrak{g}$ and $\mathfrak{g}^\ast$ induced by a   $\gt$-structure $\varphi=h\cdot\varphi_0$ are  $g=(h^{-1})^\ast g_0$ and $g=h^\ast g_0$, respectively, where $g_0$ is the inner product induced by $\varphi_0$. So, for $\alpha\in \Lambda^k(\mathfrak{g})^\ast$ we have
        \begin{align*}
                \alpha\wedge\ast\alpha=&g(\alpha,\alpha)\vol\\
                                      =&(h^\ast g_0)(\alpha,\alpha)(h^{-1})^\ast\vol_0\\
                                      =&(h^{-1})^\ast(g_0(h^\ast\alpha,h^\ast\alpha)\vol_0)\\
                                      =&\alpha\wedge (h^{-1})^\ast\ast_0h^\ast\alpha,
        \end{align*}
        which gives the first claimed relation. In particular,
        $$
         \ast\psi=(h^{-1})^\ast\ast_0 h^\ast\psi=(h^{-1})^\ast\ast_0\psi_0=h\cdot\varphi_0=\varphi.
        $$
        Applying again the first relation to the operator $d^\ast=(-1)^{7k}\ast d\ast$, we have $d^\ast=(h^{-1})^\ast\circ d^{\ast_0}\circ h^\ast$, which yields the claim because $d$ commutes with the pullback $h^\ast$.
\end{proof}
As consequence of the above Lemma, we can  relate $Q_\psi\in \mathfrak{q}_\psi$ to $Q_0\in \mathfrak{q}_{\psi_0}$:
\begin{align*}
        \theta(Q_\psi)\psi=&\Delta_\psi\psi=\Delta_\psi((h^{-1})^\ast\psi_0)=(h^{-1})^\ast(\Delta_0\psi_0)\\
                          =&(h^{-1})^\ast\theta(Q_0)\psi_0=(h^{-1})^\ast\theta(Q_0)h^\ast\psi\\
                          =&(h^{-1})^\ast\frac{d}{dt}\big(e^{tQ_0}\cdot(h^{-1}\cdot\psi)\big)|_{t=0}=\frac{d}{dt}\big((he^{tQ_0}h^{-1})\cdot\psi)\big)|_{t=0}\\
                          =&\frac{d}{dt}\big((e^{t\Ad(h)Q_0})\cdot\psi)\big)|_{t=0}=\theta(\Ad(h)Q_0)\psi,
\end{align*}
since $\mathfrak{g}_{\psi}\cap \mathfrak{q}_\psi=0$. Therefore, 
\begin{equation}
\label{Adjoint_relation}
        Q_\psi=\Ad(h)Q_0.
\end{equation}

We will address the flow \eqref{Lauret_flow} in the particular case  $(M,\gamma_t)=(G,\psi_t)$ and $q=-\Delta_{\psi_t}$, i.e. under the Laplacian co-flow \eqref{Laplacia_coflow}. In particular, a $G$-invariant solution of the Laplacian co-flow is given by a 1-parameter family in $\mathfrak{g}$ solving 
\begin{equation}\label{edo}
        \frac{d}{dt}\psi_t=-\Delta_t\psi_t.
\end{equation}  
Writing $\psi_t=:h_t^{-1}\cdot\psi_0$ for $h_t\in \Gl(\mathfrak{g})$, we have  
\begin{align*}
        \frac{d}{dt}\psi_t=&\psi_0(h_t'\cdot,h_t\cdot,h_t\cdot,h_t\cdot)+\psi_0(h_t\cdot,h_t'\cdot,h_t\cdot,h_t\cdot)+\psi_0(h_t\cdot,h_t\cdot,h_t'\cdot,h_t\cdot)+\psi_0(h_t\cdot,h_t\cdot,h_t\cdot,h_t'\cdot)\\
                           =&\psi_t(h_t^{-1}h_t'\cdot,\cdot,\cdot,\cdot)+\psi_t(\cdot,h_t^{-1}h_t'\cdot,\cdot,\cdot)+\psi_t(\cdot,\cdot,h_t^{-1}h_t'\cdot,\cdot)+\psi_t(\cdot,\cdot,\cdot,h_t^{-1}h_t'\cdot)\\
                           =&-\theta(h_t^{-1}h_t')\psi_t,
\end{align*}
thus the evolution of $h_t$ under the flow \eqref{edo} is given by
\begin{equation}\label{evolution_endomorphism}
        \frac{d}{dt}h_t=h_tQ_t.
\end{equation}

\section{Lie bracket flow}

The \emph{Lie bracket flow} is a dynamical system defined on the variety of Lie algebras, corresponding to  an invariant geometric flow under a natural change of variables. It is introduced in \cite{Lauret} as a tool for the study of regularity and long-time behaviour  of solutions.

For each $h\in \Gl(\mathfrak{g})$, consider the following Lie bracket in $\mathfrak{g}$: 
\begin{equation}
\label{bracket_action}
        \mu=[\cdot, \cdot]_h
        :=h\cdot[\cdot,\cdot]
        =h[h^{-1}\cdot,h^{-1}\cdot].
\end{equation}
Indeed, $ (\mathfrak{g},[\cdot,\cdot]) \xrightarrow{h}(\mathfrak{g},\mu)$ defines a Lie algebra isomorphism, and consequently an equivalence between invariant structures
$$
 \eta: (G,\psi_\mu)\rightarrow (G_\mu,\psi),
$$
where $G_\mu$ is the $1$-connected Lie group with Lie algebra $(\mathfrak{g},\mu)$, $\eta$ is an automorphism such that $d\eta_1=h$ and $\psi_\mu=\eta^\ast\psi$. In particular, by Lemma \ref{diffeomorphism_invariance}, $\Delta_\mu\psi_\mu=\eta^\ast\Delta_\psi\psi$, or, equivalently,  $Q_\mu=hQ_\psi h^{-1}$, by equation \eqref{Adjoint_relation}.

\begin{lemma}\cite[\S4.1]{Lauret}
        Let $\{h_t\}\subset \Gl(\mathfrak{g})$ be a solution of \eqref{evolution_endomorphism}, then the bracket $\mu_t:=[\cdot,\cdot]_{h_t}$ evolves under the flow
        \begin{equation}\label{bracket_flow}
                \frac{d}{dt}\mu_t=-\delta_{\mu_t}(Q_{\mu_t}),
        \end{equation}
in which $\delta_\mu :\End(\mathfrak{g})\rightarrow \Lambda^2(\mathfrak{g})^\ast\otimes\mathfrak{g}$ is the infinitesimal representation of the $\Gl(\mathfrak{g})$-action  \eqref{bracket_action}, defined by
$$
\delta_\mu(A):=-A\mu(\cdot,\cdot)+\mu(A\cdot,\cdot)+\mu(\cdot,A\cdot). 
$$
\end{lemma}
\begin{proof}
Setting $Q_{\mu_t}:=h_tQ_th_t^{-1}$, we compute:
\begin{align*}
        \frac{d}{dt}\mu_t=&h_t'[h_t^{-1}\cdot,h_t^{-1}\cdot]+h_t[(h_t^{-1})'\cdot,h_t^{-1}\cdot]+h_t[h_t^{-1}\cdot,(h_t^{-1})'\cdot]\\
                         =&h_t'h_t^{-1}\mu_t(\cdot,\cdot)-\mu_t(h_t'h_t^{-1}\cdot,\cdot)-\mu_t(\cdot,h_t'h_t^{-1}\cdot)\\
                         =&-\delta_{\mu_t}(h_t'h_t^{-1})=-\delta_{\mu_t}(h_tQ_th_t^{-1})=-\delta_{\mu_t}(Q_{\mu_t}),
\end{align*}
since $(h_t^{-1})'=-h_t^{-1}h_t'h_t^{-1}$.
\end{proof}

\begin{remark*}
        Notice that, if $\{h_t\}\subset \Gl(\mathfrak{g})$ solves 
        $$
        \frac{d}{dt}h_t=Q_{\mu_t}h_t,
        $$
        then $\mu_t$ solves the bracket flow \eqref{bracket_flow}.
\end{remark*}

\section{Self Similar Solutions}

We say that a $4$-form $\psi$ \emph{flows self-similarly} along the flow \eqref{Laplacia_coflow} if the solution $\psi_t$ starting at $\psi$ has the form
$\psi_t=b_tf_t^\ast \psi$, for some one-parameter families $\{f_t\} \subset \mathrm{Diff}(G)$ and time-dependent non-vanishing functions $\{b_t\}$. This is equivalent to the relation 
$$
 -\Delta\psi=\lambda\psi+\mathcal{L}_X\psi,
$$
for some constant $\lambda\in \mathbb{R}$ and a complete vector field $X$. Suppose that the infinitesimal operator defined by $\Delta\psi=\theta(Q_\psi)\psi$ had the particular form 
\begin{equation}
\label{algebraic_soliton}
        Q_\psi=cI+D \qforq c\in \mathbb{R}
        \qandq 
        D\in \Der(\mathfrak{g}).
\end{equation}
Then we have
\begin{align*}
        \theta(Q_\psi)\psi=&-4c\psi+\theta(D)\psi=-4c\psi-\frac{d}{dt}\Big((e^{tD})^\ast\psi\Big)|_{t=0}\\
                          =&-4c\psi-\mathcal{L}_{X_D}\psi, 
\end{align*}
where $X_D$ is a vector field On $\mathfrak{g}$ defined by the $1$-parameter group of automorphisms $e^{tD}\in \Aut(\mathfrak{g})$.\\
In that case, $(G,\psi)$ is a soliton for the Laplacian co-flow with
$$
  -\Delta_\psi\psi=4c\psi+\mathcal{L}_{X_D}\psi,
$$
also, $X_D$  denotes the invariant vector field on $G$ induced  by the $1$-parameter subgroup $=e^{tD}\in \Aut(\mathfrak{g})$.\\
A $\gt$-structure whose underlying $4$-form $\psi$ satisfies \eqref{algebraic_soliton} is called an \emph{algebraic soliton}, and we say that it is \emph{expanding, steady, or shrinking} if $\lambda$ is positive, zero, or negative, respectively.
\begin{lemma}\label{scaling_laplacian}
        Given $\psi_2=c\psi_1$ with $c\in \mathbb{R}^\ast$, the Laplacian operator satisfies the scaling property
        \begin{equation}
                \Delta_2\psi_2=c^{1/2}\Delta_1\psi_1
        \end{equation}
\end{lemma}

\begin{proof}
        	Notice that $c\psi_1=(c^{1/4})^4\psi_1$, then $\varphi_2=c^{3/4}\varphi_1$, $g_2=c^{1/2}g_1$ and $\vol_2=c^{7/4}\vol_1$. For a $k$-form $\alpha$ we have 
        \begin{align*}
        \alpha\wedge\ast_2\alpha=&g_2(\alpha,\alpha)\vol_2=\frac{1}{k!}\alpha_{i_1,\dots,i_k}\alpha_{j_1,\dots j_k}(g_2)^{i_1j_1}\cdots (g_2)^{i_kj_k}\vol_2\\
        =&c^{7/4-k/2}\frac{1}{k!}\alpha_{i_1,\dots,i_k}\alpha_{j_1,\dots j_k}(g_1)^{i_1j_1}\cdots (g_1)^{i_kj_k}\vol_1=c^{7/4-k/2}g_1(\alpha,\alpha)\vol_1\\
        =&c^{7/4-k/2}\alpha\wedge\ast_1\alpha.
        \end{align*}
        So, for a $k$-form $\ast_2\alpha=c^{\frac{1}{4}(7-2k)}\ast_1\alpha$. And for the Hodge Laplacian operator we have 
        	\begin{align*}
        	\Delta_2\psi_2=&d\ast_2d\ast_2\psi_2-\ast_2d\ast_2d\psi_2=cd\ast_2d\ast_2\psi_1-c\ast_2d\ast_2d\psi_1\\
        	=&c^{3/4}d\ast_2d\ast_1\psi_1-c^{1/4}\ast_2d\ast_1d\psi_1=c^{1/2}d\ast_1d\ast_1\psi_1-c^{1/2}\ast_1d\ast_1d\psi_1=c^{1/2}\Delta_1\psi_1.
        	\end{align*}
\end{proof}

The following Lemma appears in \cite{Fino,Lauret}, we write the proof for the specific case of the Laplacian co-flow. 
\begin{lemma}\label{soliton_lemma}
        If $\psi$ is an algebraic soliton with $Q_\psi=cI+D$, then $\psi_t=b_th_t^\ast\psi$ is a self-similar solution for the Laplacian co-flow \eqref{edo}, with
        	\begin{equation}
        b_t=(2ct+1)^2
        \qandq
        \quad h_t=e^{s_tD}, 
        \qforq s_t=-\frac{1}{2c}\log(2ct+1).
        \end{equation}
        Moreover, 
        \begin{equation*}
                Q_t=b_t^{-1/2}Q_\psi.
        \end{equation*}
\end{lemma}
\begin{proof}
        Applying Lemmata \ref{diffeomorphism_invariance} and \ref{scaling_laplacian}, we have
\begin{align*}
        \Delta_t\psi_t&=b_t^{1/2}h_t^\ast\Delta\psi=b_t^{1/2}h_t^\ast\theta(Q_\psi)\psi\\
        &=b_t^{1/2}h_t^\ast\Big(-4c\psi+\theta(D)\psi\Big)\\
        &=-4cb_t^{1/2}h_t^\ast\psi+\theta(b_t^{1/2}h_t^{-1}Dh_t)h_t^\ast\psi.
\end{align*}
On the other hand, 
\begin{align*}
        \frac{d}{dt}\psi_t=&b_t'h_t^\ast\psi+b_t(h_t^\ast\psi)'\\
        =&b_t'h_t^\ast\psi +b_t\theta(h_t^{-1}h_t')h_t^\ast\psi.
\end{align*}
Replacing the above expressions in \eqref{edo} and comparing terms we obtain the ODE system  
$$
\begin{cases}
b_t'=4cb_t^{1/2}, & b(0)=1 \\
b_th_t'=-b_t^{1/2}Dh_t, & h(0)=I \\
\end{cases},
$$ 
the solutions of which are as claimed.

Finally, we have
\begin{align*}
        \theta(Q_t)\psi_t&=\Delta_t\psi_t=b_t^{1/2}h_t^\ast\Delta\psi=b_t^{1/2}h_t^\ast\theta(Q_\psi)\psi\\
        &=b_t^{1/2}\theta(h_t^{-1}Q_\psi h_t)h_t^\ast\psi=\theta(b_t^{-1/2}h_t^{-1}Q_\psi h_t)\psi_t,
\end{align*}
so $Q_t=b_t^{-1/2}h_t^{-1}Q_\psi h_t$, which yields the second claim, since $Q_\psi h_t=h_tQ_\psi$.
\end{proof}
In terms of the bracket flow, we have $Q_{\mu_t}=h_tQ_th_t^{-1}=b_t^{-1/2}Q_\psi$. Then, replacing in \eqref{bracket_flow} the Ansatz 
\begin{equation}\label{bracket_flow_soliton}
\mu_t=(\frac{1}{c(t)}I)\cdot[\cdot,\cdot]=c(t)[\cdot,\cdot]
\qforq
c(t)\neq 0
\qandq 
c(0)=1,
\end{equation}
we obtain $c_t'=cb_t^{-1/2}c_t$, which has solution $c_t=e^{c.s_t}$, with $s_t$ as above.  

Indeed, there is an equivalence between the time-dependent Lie bracket given in \eqref{bracket_flow_soliton} and the corresponding soliton given in Lemma \ref{soliton_lemma}:

\begin{theorem}\cite[Theorem 6]{Lauret}
        Let $(G,\varphi)$ be a $1$-connected Lie group with an invariant $\gt$-structure. The following conditions are equivalent:
        \begin{enumerate}
                \item[\rm{(i)} ] The bracket flow solution starting at $[\cdot,\cdot]$ is given by 
                $$
                  \mu_t=(\frac{1}{c(t)}I)\cdot[\cdot,\cdot] \qforq c(t)>0, c(0)=1.
                $$ 
                \item[\rm{(ii)} ] The operator $Q_t\in\mathfrak{q}_\psi\subset\End(\mathfrak{g})$, such that $\Delta_\psi\psi=\theta(Q_\psi)\psi$, satisfies 
                $$
                  Q_\psi=cI+D, \qforq c\in\mathbb{R} 
                  \qandq 
                  D\in \Der(\mathfrak{g}).
                $$
        \end{enumerate}
\end{theorem}

\section{Example of a co-flow soliton }
We now apply the previous theoretical framework to construct an explicit co-flow soliton from a natural Ansatz. Let $\mathfrak{g}=\mathbb{R}\times_\rho \mathbb{R}^6$ be the Lie algebra defined by $\rho(t)=\expo(tA)\in \Aut(\mathfrak{g})$, with
$$
A=\left(\begin{array}{ccc|ccc}
          &  &  &  &  & 1 \\
          &  &  &  & 0&   \\
          &  &  & 1&  &   \\
         \hline
          &  & 1&  &  &   \\
          & 0&  &  &  &   \\
         1&  &  &  &  &   \
\end{array}
\right).
$$
The canonical $\SU(3)$-structure on $\mathbb{R}^6$ with respect to the orthonormal basis $\{e_1,e_6,e_2,e_5,e_3,e_4\}$ is
$$
  \omega=e^{16}+e^{25}+e^{34}, \quad \rho_+=e^{123}+e^{145}+e^{356}-e^{246}
$$
and the standard complex structure of $\mathbb{R}^6$ is
$$
J=\begin{pmatrix}
0 & -I_{3}\\
I_{3} & 0
\end{pmatrix}.
$$
We also have the natural $3$-form 
$$
\rho_-:=J\cdot \rho_+=-e^{135}+e^{124}+e^{236}+e^{456}.
$$ 
The structure equations of $\mathfrak{g}^\ast$ with respect to the dual basis of $\{e_1,e_6,e_2,e_5,e_3,e_4,e_7\}$ are
$$
  de^1=e^{67}, \quad de^6=e^{17}, \quad de^3=e^{47}, \quad de^4=e^{37}, \quad de^j=0 \qforq j=2,5.
$$
From the above, we have 
$$
d\omega=0, \quad d\rho_+=2(e^{1357}+e^{4567}), \qandq d\rho_-=2(e^{2467}+e^{1237}).
$$
There is a  natural co-closed $\gt$-structure on $\mathfrak{g}$, given by 
\begin{equation}
        \varphi:=\omega\wedge e^7 -\rho_-=e^{167}+e^{257}+e^{347}+e^{135}-e^{124}-e^{236}-e^{456},
\end{equation}
with dual $4$-form
\begin{equation}
\psi=\ast\varphi=\frac{\omega^2}{2}+\rho_+\wedge e^7=e^{1256}+e^{1346}+e^{2345}+e^{1237}+e^{1457}+e^{3567}-e^{2467}.
\end{equation}
Clearly $\tau_1=0$ and $\tau_2=0$, and
$$
 d\varphi=-d\rho_-=-2(e^{2467}+e^{1237})=\ast\tau_3,
$$
since $d\varphi\wedge \varphi=0$, i.e. $\tau_0=0$. Therefore, using \eqref{Bryan_map},  $$
\tau_3=2(e^{135}+e^{456}) 
\quad\text{or, alternatively,}\quad 
\tau_{27}=(e^1)^2+(e^3)^{2}-((e^4)^2+(e^6)^2).
$$
The Laplacian of $\psi$ is
\begin{align*}
        \Delta\psi=&d\ast d\ast \psi+\ast d\ast d\psi=d\ast d\varphi\\
                  =&d\tau_ 3=4(e^{1457}+e^{3567}).
\end{align*}
Consider the derivation $D=\diag(a,b,c,c,d,a,0)\in \Der(\mathfrak{g})$, and take the vector field on $\mathfrak{g}$ 
$$
X_D(x)=\frac{d}{dt}(\expo(tD)(x)), \qforq x\in \mathfrak{g}.
$$
Then we have
\begin{align*}
   \mathcal{L}_{X_D}\psi=&\;\frac{d}{dt}(\expo(-tD)^\ast \psi)|_{t=0}=-\theta(D)\psi\\
                        =&\;(2a+b+d)e^{1256}+(2a+2c)e^{1346}+(b+2c+d)e^{2345}+(a+b+c)e^{1237}\\
                         &+(a+c+d)e^{1457}+(a+c+d)e^{3567}-(a+b+c)e^{2467}.  
\end{align*}
From the soliton equation $-\Delta\psi=\mathcal{L}_{X_D}\psi+\lambda\psi$, we obtain a system of linear equations

$$
\left\{ \begin{array}{rcc}
        2a+b+d+\lambda &=& 0 \\
        2a+2c+\lambda  &=& 0 \\
        a+b+c+\lambda  &=& 0 \\
        a+c+d+\lambda  &=&-4
\end{array} \right.,
$$
which has solution $D=\diag(2,4,2,2,0,2,0)$ and $\lambda=-8$. In particular, for the matrix $Q_\psi=D+\frac{\lambda}{4}I_{7}$, we have $\Delta\psi=\theta(Q_\psi)\psi$.
By Lemma \ref{soliton_lemma},  the functions
$$
c(t)=(1-4t)^2 \qandq s(t)=\frac{1}{4}\log(1-4t) \qforq \frac{1}{4}>t,
$$ 
 
yield the family of $4$-forms $\{\psi_t=c(t)(f(t)^{-1})^\ast\psi\}$, where
\begin{align*}
  f(t)^{-1}&=\exp(-s(t)D)\\
  &=\diag((1-4t)^{-1/2},(1-4t)^{-1},(1-4t)^{-1/2},(1-4t)^{-1/2},1,(1-4t)^{-1/2},1).
  \end{align*}
Hence,
\begin{equation}\label{4form_solution}
\psi_t=e^{1256}+e^{1346}+e^{2345}+e^{1237}+(1-4t)(e^{1457}+e^{3567})-e^{2467}
\end{equation}
defines a soliton of the Laplacian co-flow:
$$
\frac{d\psi_t}{dt}=-4(e^{1457}+e^{3567})=-c(t)^{1/2}(f(t)^{-1})^\ast\Delta\psi=-\Delta_t\psi_t.
$$

\begin{corollary}
\label{tensors}
        The relevant geometric structures associated to the $4$-form given in \eqref{4form_solution} are:
        \begin{enumerate}
        \item[\rm{(i)} ] the $\gt$-structure
        $$
          \varphi_t=c(t)^{1/4}(e^{167}+e^{257}+e^{347}+e^{135}-e^{456})-c(t)^{-1/4}(e^{124}+e^{236});
        $$
        \item[\rm{(ii)}] the $\gt$-metric
        $$
          g_t=(e^1)^2+(e^3)^2+(e^4)^2+(e^6)^2+c(t)^{-1/2}(e^2)^2+c(t)^{1/2}((e^5)^2+(e^7)^2);
        $$
        \item[\rm{(iii)}] the volume form
        $$
          \vol_t=c(t)^{1/4}\vol_\psi;
        $$
        \item[\rm{(iv)}] the torsion form and the full torsion tensor
        $$
          \tau_3(t)=2(e^{135}+e^{456}) \qandq T(t)=c(t)^{-1/4}\Big(-(e^1)^2-(e^3)^2+(e^4)^2+(e^6)^2\Big);
        $$
        \item[\rm{(v)}] the Ricci tensor and the scalar curvature
        $$
          \ricci(g_t)=-4c(t)^{-1/2}(e^7)^2
        \qandq R_t=-\frac{1}{2}|\tau_3(t)|^2=-4c(t)^{-1/2};
        $$
        \item[\rm{(vi)}] the bracket flow solution
        $$
          \mu_t=c(t)^{-1/4}[\cdot,\cdot].
        $$
\end{enumerate}
\end{corollary}    

\begin{remark}
        \quad
        \begin{enumerate}[(1)]
                \item  
                From Corollary \ref{tensors} $\rm{(iv)}$ and $\rm{(v)}$, if $t\rightarrow -\infty$ then $\ricci(g_t)\rightarrow 0$, $T(t)\rightarrow 0$ and $\mu_t\rightarrow 0$. Since $G_{\mu_t}$ is solvable for each $t$ \cite[Proposition 6]{Lauret}, $(G_{\mu_t},\psi)$ smoothly converges to the flat $\gt$-structure $(\mathbb{R}^7,\varphi_0)$. 
                \item 
                Since $\ricci(g_\psi)=\diag (0,0,0,0,0,0,-4)+4I_7 \in \Der(\mathfrak{g})$, the metric $g_\psi$ is a shrinking Ricci soliton (cf.  \cite{arroyo}).
                \end{enumerate}
\end{remark}

\section*{Afterword}

We would like to conclude with two questions for future work. 

\begin{enumerate}
\item 
If the cover Lie group $G$ admits a co-compact discrete subgroup $\Gamma$, it would be interesting to determine whether the corresponding co-closed $\gt$-structure on the compact quotient $G/\Gamma$ is also a Laplacian co-flow soliton.

\item
                 When the full torsion tensor $T=-\tau_{27}$ is traceless symmetric, the scalar curvature of the corresponding  $\gt$-metric is nonpositive,  and it vanishes if, and only if, the structure is torsion-free (c.f. \cite[(4.28)]{Bryan2} or \cite[(4.21)]{Karigianis2}). This fact was first pointed out by Bryant for a closed $\gt$-structure, in order to explain  the absence of closed Einstein $\gt$-structures (other than Ricci-flat ones) on compact $7$-manifolds,   giving rise to the concept of \emph{extremally Ricci-pinched closed $\gt$-structure} \cite[Remark 13]{Bryan2}.
Later on,
                Fernández et al. showed that a 7-dimensional (non-flat) Einstein solvmanifold $(S,g)$ cannot admit any left-invariant co-closed $\gt$-structure $\varphi$ such that  $g_\varphi=g$ \cite{Fernandez2}.
                
In that context, it would be interesting to study pinching phenomena for the Ricci curvature of solvmanifolds with a co-closed (non-flat) left-invariant $\gt$-structure and traceless torsion. In our present construction, for instance, we can see from Corollary \ref{tensors} that
                $$
                  F(t)=\frac{R_t^2}{|\ricci(g_t)|^2}=1.
                $$
\end{enumerate}


\end{document}